\documentclass[10pt]{amsart}
\usepackage{egastyle}

\usepackage{amssymb,amsmath,amscd,amsthm,amsfonts,wasysym,mathrsfs,amsxtra,stmaryrd,xcolor,xypic}
\usepackage{mathtools}%
\usepackage{enumerate}%
\usepackage{tikz-cd}%
\usepackage{verbatim}
\usepackage{graphicx,array,url}
\usepackage[vcentermath]{genyoungtabtikz}
\usepackage{tikz}
\usepackage{verbatim}

\input xy
\xyoption{all}

\renewcommand{\phi}{\varphi}

\newcommand{\image}{\mathrm{im}\,}
\newcommand{\kernel}{\mathrm{ker}\,}

\newcommand{\Hom}{\mathrm{Hom}}
\newcommand{\dimension}{\mathrm{dim}\,}
\newcommand{\codimension}{\mathrm{codim}\,}

\newcommand{\supp}{\mathrm{supp}\,}

\newcommand{\Ac}{\mathcal{A}}
\newcommand{\Bc}{\mathcal{B}}
\newcommand{\Cc}{\mathcal{C}}
\newcommand{\Fc}{\mathcal{F}}
\newcommand{\Tc}{\mathcal{T}}

\newcommand{\Coh}{\mathrm{Coh}}

\newcommand{\arinj}{\ar@{^{(}->}}
\newcommand{\arsurj}{\ar@{->>}}
\newcommand{\areq}{\ar@{=}}

\newcommand{\wh}{\widehat}

\newcommand{\ch}{\mathrm{ch}}

\newcommand{\scalea}{\scalebox{0.5}}

\newcommand{\al}{{\alpha}}

\newcommand{\Wi}{W_{1,\Phi}}

\usepackage{scalerel,stackengine}
\stackMath
\newcommand\reallywidehat[1]{%
\savestack{\tmpbox}{\stretchto{%
  \scaleto{%
    \scalerel*[\widthof{\ensuremath{#1}}]{\kern-.6pt\bigwedge\kern-.6pt}%
    {\rule[-\textheight/2]{1ex}{\textheight}}
  }{\textheight}%
}{0.5ex}}%
\stackon[1pt]{#1}{\tmpbox}%
}
\makeatletter
\newcommand\makebig[2]{%
  \@xp\newcommand\@xp*\csname#1\endcsname{\bBigg@{#2}}%
  \@xp\newcommand\@xp*\csname#1l\endcsname{\@xp\mathopen\csname#1\endcsname}%
  \@xp\newcommand\@xp*\csname#1r\endcsname{\@xp\mathclose\csname#1\endcsname}%
}
\makeatother

\makebig{biggg} {3.0}
\makebig{Biggg} {3.5}
\makebig{bigggg}{4.0}
\makebig{Bigggg}{4.5}

\begin{document}

\title[A framework for torsion theory computations]{A framework for  torsion theory computations on elliptic threefolds}

\author[David Angeles]{David Angeles}
\address{Department of Statistics \\ The Ohio State University \\ 1958 Neil Ave \\
Columbus OH 43210 \\ USA}
\email{angeles.6@osu.edu}

\author[Jason Lo]{Jason Lo}
\address{Department of Mathematics \\California State University, Northridge\\18111 Nordhoff Street\\Northridge CA 91330 \\USA}
\email{jason.lo@csun.edu}

\author[Courtney van der Linden]{Courtney M van der Linden}
\address{Department of Mathematics \\California State University, Northridge\\18111 Nordhoff Street\\Northridge CA 91330 \\USA}
\email{courtney.vanderlinden.727@my.csun.edu}

\thanks{Partially supported by NSF-DMS 1247679 grant PUMP: Preparing Undergraduates through Mentoring towards PhD's}

\keywords{torsion pair, torsion theory, t-structure,  elliptic threefold}
\subjclass[2010]{Primary 14F05; Secondary: 18E40, 14J30}

\begin{abstract}
We give a list of statements on the geometry of elliptic threefolds phrased only in the language of topology and homological algebra.  Using only notions from topology and homological algebra, we recover existing results and prove new results on torsion pairs  in the category of coherent sheaves on an elliptic threefold.
\end{abstract}

\maketitle
\tableofcontents


\section{Introduction}

Given a smooth projective variety $X$, in order to study a moduli space of geometric objects on $X$, one often needs to fix a stability condition.  In particular, when the geometric objects are coherent sheaves or chain complexes of coherent sheaves, a stability condition (e.g.\ slope stability, tilt stability, or Bridgeland stability) is defined only after one fixes a t-structure on the derived category of coherent sheaves $D^b(\Coh (X))$ on $X$.  In other words, a t-structure is a precursor to a stability condition.  

Since every torsion class in the abelian category $\Coh (X)$ defines a t-structure on $D^b(\Coh (X))$, the study of torsion classes in $\Coh (X)$ and their relations can yield useful information on stability conditions on $X$, and ultimately moduli spaces of objects on $X$.  There are certain torsion classes in $\Coh (X)$, in fact, that can be studied with a minimal amount of knowledge of algebraic geometry.  Specifically, one can begin with a list of `axioms' in which algebro-geometric information is already encoded, and proceed to deduce consequences of these axioms using basic topology and homological algebra.  These consequences then correspond to statements in algebraic geometry.

In this paper, we provide such a list of axioms and prove a series of results using only notions from topology and homological algebra.  Some of these results were essential to the proof of the Harder-Narasimhan property of limit tilt stability on elliptic threefolds in \cite{Lo14,Lo15}; other results in this article are new.  All the results in this paper are results on t-structures on $D^b(\Coh (X))$ when $X$ is an elliptic threefold.


We now describe the results of this article. For a smooth projective threefold $X$ with a Weierstra{\ss} fibration, we  define subcategories $C_{ij}$ of $\Coh (X)$ (see Section \ref{sec:computations}). We set a category $C_{ij}$ to be empty if it does not appear in \eqref{eq:Cij3Dpic-1}: 

\begin{equation}\label{eq:Cij3Dpic-1}
\xymatrix{
& & C_{40} \ar@{-}[r] & C_{50} \ar@{-}[d] \\
& C_{20} \ar@{-}[ur] \ar@{-}[r] & C_{30} \ar@{-}[ur] \ar@{-}[d]& C_{51} \ar@{-}[d] \\
C_{00} \ar@{-}[r] \ar@{-}[ur] & C_{10} \ar@{-}[ur] \ar@{-}[d] & C_{31} \ar@{-}[d]& C_{52} \\
& C_{11} \ar@{-}[d] & C_{32} & \\
& C_{12} & & 
}
\end{equation}
For $n=0,2,4$, we define 
\[
  \Tc_{n0} = \langle C_{ij} : i \leq n, 0 \leq j \leq 2 \rangle
\]
while for $n= 1, 3, 5$, we define 
\begin{align*}
  \Tc_{n0} &= \langle \Tc_{(n-1),0}, C_{n0} \rangle, \\
  \Tc_{n1} &= \langle \Tc_{n0}, C_{n1} \rangle, \\
  \Tc_{n2} &= \langle \Tc_{n1}, C_{n2} \rangle
\end{align*}
and for $2\leq i \leq 5$ we set $\mathcal{F}_i=\langle C_{00}, C_{10},\ldots, C_{i0} \rangle$ (see Definitions \ref{Jazz} and \ref{Jonathan}).  The results of this article can be summarised in the following theorem:

\begin{thm}\label{thm:main}
Let $X = C \times B$ be the product of an elliptic curve $C$ and a K3 surface $B$ of Picard rank one.   The following seventeen categories are all torsion classes in the category of coherent sheaves on $X$:
\begin{itemize}
\item $\Tc_{00}, \Tc_{20}, \Tc_{40}$.
\item $\Tc_{ij}$ for $i=1,3,5$ and $0 \leq j \leq 2$.
\item $\Fc_i$ for $2 \leq i \leq 5$.
\item $\langle C_{00}, C_{20}\rangle$.
\end{itemize}
\end{thm}

The structure of this paper is as follows: in Section \ref{sec:prelim}, we collect the key notions from homological algebra we need.  In Section \ref{se:propertieslist}, we give a list of `Properties' that have algebraic geometry encoded in them; these Properties are treated as axioms in this paper.  In Section \ref{sec:computations}, we use only the Properties in Section \ref{se:propertieslist} and concepts from topology and homological algebra to prove a series of lemmas, which together give Theorem \ref{thm:main}.

\section{Preliminaries on homological algebra}\label{sec:prelim}

We assume that the reader is familiar with the basic properties of abelian categories, triangulated categories and exact functors of triangulated categories.

\paragraph[Torsion pairs] Suppose $\Ac$ is an abelian category.  A pair of full subcategories $(\Tc, \Fc)$ of $\Ac$ is called a torsion pair (or a torsion theory) if \cite{HRS}:
\begin{itemize}
\item Every object $E \in \Ac$ fits in a short exact sequence in $\Ac$
\[
0 \to E' \to E \to E'' \to 0
\]
for some $E' \in \Tc$ and $E'' \in \Fc$.
\item For any $E' \in \Tc$ and any $E'' \in \Fc$, we have $\Hom_{\Ac}(E',E'')=0$.
\end{itemize}
Given a torsion pair $(\Tc, \Fc)$ in $\Ac$, we will refer to $\Tc$ (resp.\ $\Fc$) as the torsion class (resp.\ torsion-free class) in the torsion pair.  We will say a subcategory $\mathcal{C}$ of $\Ac$ is a torsion class (resp.\ torsion-free class) in $\Ac$ if it is the torsion class (resp.\ torsion-free class) in a torsion pair.

If $(\Tc, \Fc)$ is a torsion pair in $\Ac$, then $\Tc, \Fc$ are both extension-closed, and $\Tc$ is closed under quotient in $\Ac$ while $\Fc$ is closed under subobject in $\Ac$.

In a noetherian abelian category, such as the category of coherent sheaves on an algebraic variety, a torsion class can be recognised via the follow lemma:

\begin{lem}\cite[Lemma 1.1.3]{Pol}\label{lem:Pol}
Let $\mathcal{C}$ be a noetherian abelian category.  Then any full subcategory $\Tc \subseteq \mathcal{C}$ closed under quotients and extensions is a torsion class.
\end{lem}

For any subcategory $\Cc$ of an abelian category $\Ac$, we will often write $\Cc^\circ$ to denote the subcategory of $\Ac$
\[
  \Cc^\circ = \{ E \in \Ac : \Hom_{\Ac}(E',E)=0 \text{ for all } E' \in \Cc\}.
\]
For example, if $\Tc$ is a torsion class in an abelian category $\Ac$, then $\Tc^\circ$ coincides with the corresponding torsion-free class.

Given two subcategories $\Ac, \Bc$ of a category $\mathcal{C}$, we often write $\Hom_{\mathcal C} (\Ac, \Bc)=0$ to mean that $\Hom_{\mathcal C} (A,B)=0$ for all $A \in \Ac, B \in \Bc$.

Given an abelian category $\Ac$, a extension-closed subcategory $\Cc$ of $\Ac$ is called a Serre subcategory if, given any object $E$ of $\Cc$ and any subobject or quotient $E'$ of $E$, the object $E'$ also lies in $\Cc$.  By Lemma \ref{lem:Pol}, any Serre subcategory of a noetherian abelian category is a torsion class.

\paragraph[Derived category] Given an abelian category $\Ac$, the bounded derived category $D^b(\Ac)$ of $\Ac$ is a category in which the objects $E$ are chain complexes of objects in $\Ac$.  That is, if $E$ is an object of $D^b(\Ac)$, then $E$ is represented by a diagram in $\Ac$
\begin{equation}\label{eq:chaincomplex}
  \cdots \to E^{i-1} \overset{d^{i-1}}{\to} E^i \overset{d^i}{\to} E^{i+1} \overset{d^{i+1}}{\to} \cdots
\end{equation}
where  each  composition $d^{i+1}d^i$ is the zero map, and the $E^i$ are zero for all but finitely many $i$.  The morphisms $d^i$ are called the differentials, and we sometimes write $d^i_E$ for $d^i$ to emphasise that the differentials are of the complex $E$.   For any integer $i$, we will refer to the degree-$i$ cohomology of $E$ with respect to the standard t-structure on $D^b(\Ac)$, defined as
\[
  H^i(E) = \frac{\kernel (d^i)}{\image (d^{i-1})},
\]
 simply as the cohomology of $E$.   We will also refer to $D^b(\Ac)$ simply as the derived category of $\Ac$.

 The derived category $D^b(\Ac)$ is a triangulated category. The reader may refer to references such as \cite{huybrechts2006fourier,neeman2014triangulated,weibel1995introduction} for basic properties of a triangulated category. The main  properties of a triangulated category that will be used in the computations in this article include:
 \begin{itemize}
  \item There is a shift functor $[1] : D^b(\Ac)\to D^b(\Ac)$ that takes an object $E$ as in \eqref{eq:chaincomplex} to the object $E[1]$ where $(E[1])^i = E^{i+1}$ and $d_{E[1]}^i= d^{i+1}_E$ for each $i$.
   \item There is an embedding functor from $\Ac$ into $D^b(\Ac)$ that takes every object $M$ of $\Ac$ to the complex `concentrated in degree 0', i.e.\ the complex \eqref{eq:chaincomplex} where $E^0=M$, all the other $E^i$ are zero, and all the morphisms $d^i$ are zero maps.
   \item Via the embedding $\Ac \to D^b(\Ac)$ above, for any $A, B \in \Ac$ we have
       \[
       \Hom_{\Ac}(A,B) \cong \Hom_{D^b(\Ac)}(A,B).
       \]
  \item For every short exact sequence $0 \to A \overset{\alpha}{\to} B \overset{\beta}{\to} C \to 0$ in $\Ac$, there is a corresponding exact triangle in $D^b(\Ac)$
     \[
      \cdots \to A \overset{\alpha}{\to} B \overset{\beta}{\to} C \to A [1] \overset{\alpha[1]}{\to} B[1] \to \cdots .
     \]
 \item Given any exact triangle
 \[
   \cdots \to E \to F \to G \to E[1] \to \cdots
    \]
    in $D^b(\Ac)$, there is a corresponding long exact sequence of cohomology in $\Ac$
     \[
     \cdots \to H^{i-1}(G) \to H^i(E) \to H^i(F) \to H^i(G) \to H^{i+1}(E) \to \cdots .
     \]
 \end{itemize}

\section{Properties encoding algebraic geometry}\label{se:propertieslist}

In this section, we give a list of statements on the Zariski topology of algebraic varieties, Fourier-Mukai functors between dual elliptic fibrations,  and Chern classes.  These statements are listed as `Properties', and  will be treated as `axioms' to perform the computations in Section \ref{sec:computations}.

\paragraph[Properties of dimension]  For any smooth projective variety $Y$, the coherent sheaves on $Y$ form an abelian category, which we denote by $\Coh (Y)$ or $\Ac_Y$.  For any $E \in \Ac_Y$, there is an associated topological subspace $\supp(E)$ of $Y$ with respect to the Zariski topology on $Y$.  Given any morphism of projective varieties $p : Y \to Z$, it induces a continuous map $Y \to Z$ with respect to the Zariski topology, which we also denote by $p$.  For an object $E \in D^b(\Ac_Y)$, we define the support of $E$ to be $\supp(E) = \cup_{i \in \mathbb{Z}}\, \supp (H^i(E))$.

Given an object $E \in D^b(\Ac_Y)$, we also refer to the dimension of $\supp (E)$ simply as the dimension of $E$.  Properties of dimension include:
\begin{itemize}
\item \textbf{Property D0.} For every projective variety $Y$ of dimension $n$ and any topological subspace $W$ of $Y$ with respect to the Zariski topolgy, we can define the \emph{dimension} of $W$,  a nonnegative integer denoted by $\dimension W$.  The \emph{codimension} of $W$ is defined to be $n-\dimension W$. For any $E \in \Ac_Y$, we define the dimension (resp.\ codimension) of $E$, denoted $\dimension E$ (resp.\ $\mathrm{codim}\, E$), to be that of the support of $E$.

\item \textbf{Property D1.} For any smooth projective variety $Y$ and any short exact sequence in $\Ac_Y$
    \[
      0 \to M \to E \to N \to 0,
    \]
    we have $\supp (E) = \supp(M) \cup \supp (N)$, and hence $\dimension E = \max{\{ \dimension M, \dimension N\}}$.  
\item \textbf{Property D2.} For any morphism of smooth projective varieties $p : Y \to Z$ and any $E \in \Ac_Y$, we have $\dimension (\pi (\supp E))\leq \dimension E$.  When $p$ is flat of relative dimension $n$, we have
    \[
      (\dimension E - \dimension (\pi (\supp E))) \leq n.
    \]
\item \textbf{Property D3.} For any morphism of smooth projective varieties $p : Y \to Z$ and any short exact sequence in $\Ac_Y$
    \[
      0 \to M \to E \to N\to 0,
    \]
    we have
    \[
      \dimension (\pi (\supp E)) = \max{\{ \dimension (\pi (\supp M)), \dimension (\pi (\supp N))\}}.
    \]
\end{itemize}

We also have the following property on the support of a  coherent sheaf:
\begin{itemize}
\item \textbf{Property Z1.} Suppose $Y$ is a  projective variety and $E \in \Ac_Y$.  Let $W = \supp E$, and let $W_1, \cdots, W_m$ denote the irreducible components of $W$ with respect to the Zariski topology.   Then for any $1 \leq i \leq m$, there exists a short exact sequence in $\Ac_Y$
\[
0 \to K \to E \to E|_{W_i} \to 0
\]
where $\supp (E|_{W_i}) = W_i$, and $\supp (K) = \cup_{j \neq i} W_j$.  We refer to $E|_{W_i}$ as  the \emph{restriction of $E$ to $W_i$}.
\end{itemize}

\paragraph For any smooth projective variety $Y$ and any nonnegative integer $d$, we define
   \begin{equation*}
     \Ac_Y^{\leq d} = \{ E \in \Ac_Y : \dimension (E)   \leq d \}.
   \end{equation*}
For any morphism of smooth projective varieties $p : Y \to Z$ and integers $0\leq e \leq d$, we define
   \begin{align*}
     \Ac^d (p)_e &= \{ E \in \Ac_Y : \dimension E = d, \dimension (p (\supp E))=e\},\\
     \Ac (p)_{\leq e} &= \{ E \in \Ac_Y : \dimension (p (\supp E))\leq e\}.
   \end{align*}
We also write $\Ac (p)_0 = \Ac (p)_{\leq 0}$.

\paragraph[Our elliptic threefold] Unless \label{para:ellip3folddef} otherwise stated, throughout this article, we will write $X$ to denote a smooth projective threefold that admits a Weierstra{\ss} elliptic fibration $\pi : X \to B$ in the sense of \cite[Section 6.2]{FMNT}.  That is, $B$ is a smooth projective surface, the morphism $\pi$ is flat of relative dimension 1, all the fibers of $\pi$ are integral and Gorenstein of arithmetic genus 1 (with the generic fiber being a smooth elliptic curve), and there exists a section $\sigma : B \hookrightarrow X$ such that $\Theta := \sigma (B)$ does not meet any singular point of any fiber of $\pi$.  A coherent sheaf on $X$ that is set-theoretically supported on a finite union of fibers of $\pi$ is called a fiber sheaf, and so $\Ac (\pi)_0$ is the category of fiber sheaves on $X$.

For a proper morphism of varieties of relative dimension 1 (e.g.\ an elliptic fibration), we have the following description of 1-dimensional closed subvarieties of the domain, which follows from \cite[Lemma 3.15]{Lo11}:
\begin{itemize}
\item \textbf{Property Z2.} Suppose $p: Y \to Z$ is a proper morphism of varieties of relative dimension 1, and $W$ is an irreducible 1-dimensional  closed subset of $Y$ with respect to the Zariski topology.  Then $W$ is either of the following two types:
    \begin{itemize}
    \item[(a)] $W$ is contained in $p^{-1}(a)$ for some $a \in Z$;
    \item[(b)] for any $b \in Z$, the intersection $W \cap p^{-1}(b)$ is a finite number of points.
    \end{itemize}
\end{itemize}

\subparagraph[$\Ac^{\leq 1}_h$, the category of sheaves with horizontal supports] When \label{para:horizontalsheaves} $\pi : X \to B$ is a Weierstra{\ss} elliptic threefold, we define $\Ac^{\leq 1}_h$ to be the full subcategory of $\Ac^{\leq 1}_X$ consisting of 1-dimensional sheaves $E$ such that the 1-dimensional irreducible components of $\supp (E)$ are all of type (b) in Property Z2.  By construction, every nonzero $E \in \Ac^{\leq 1}_h$  lies in $\Ac^1 (\pi)_1$.  Note that sheaves in $\Ac^{\leq 1}_h$ are not necessarily pure 1-dimensional sheaves.

\paragraph[Autoequivalence of $D^b(X)$]{For a Weierstra{\ss} elliptic threefold $\pi : X \to B$, we have a pair of  autoequivalences of the derived category $D^b(\Ac_X)$
\[
  \Phi, \wh{\Phi} : D^b(\Ac_X) \to D^b(\Ac_X).
\]
These two functors $\Phi, \wh{\Phi}$ are  relative Fourier-Mukai transforms with kernels given by relative Poincar\'{e} sheaves (see \cite[Section 6.2.3]{FMNT} or \cite{BMef} for the precise definitions).  For any object $E \in D^b(\Ac_X)$, if $\Phi E \in \Ac_X [-i]$ for some $i$, then we say $E$ is $\Phi$-WIT$_i$ and write $\wh{E}$ to denote any object satisfying $\Phi E \cong \wh{E} [-i]$ in $D^b(\Ac_X)$.  In this case, the object $\wh{E}$ is unique up to isomorphism in $D^b(\Ac_X)$.  The notion of $\wh{\Phi}$-WIT$_i$ can similarly be defined.}

The functor $\Phi$ satisfies the following properties, with analogous properties satisfied by $\wh{\Phi}$:

\begin{itemize}
\item \textbf{Property A1.} For any $E \in \Ac_X$, we have $H^i (\Phi E)=0$ for all $i \neq 0, 1$.
\item \textbf{Property A2.} $\wh{\Phi}\Phi \cong \mathrm{id}_{D^b(\Ac_X)}[-1] \cong \Phi \wh{\Phi}$. 
\item \textbf{Property A3.} For any $E \in \Ac_X$, we have
\[
  \dimension (\pi (\supp E)) = \dimension (\pi (\supp (\Phi E))).
\]
\item \textbf{Property A4.}
For any $E \in D^b(\Ac_X)$, we have $H^i(E)=0$ for all $i \neq j$ if and only if $E$ is isomorphic in $D^b(\Ac_X)$ to some object in $\Ac_X[-j]$.   
\end{itemize}

Properties A1 holds because $\pi$ has relative dimension 1, and the kernel of $\Phi$ is a sheaf \cite[6.1.1]{FMNT}.  Property A2 is \cite[Theorem 6.18]{FMNT}.  Property A3 holds because $\Phi, \wh{\Phi}$ are relative integral functors that satisfy the base change property \cite[Proposition 6.1]{FMNT}, while  Property A4 holds for the derived category of any abelian category.

For $i=0,1$, we define $W_{i,\Phi}$ to be the subcategory of $\Ac_X$ consisting of all the $\Phi$-WIT$_i$ objects in $\Ac_X$.

\begin{lem}\label{lem:W01transform}
For any $E \in W_{i,\Phi}$ where $0\leq i \leq 1$, we have $\wh{E} \in W_{1-i,\Phi}$.
\end{lem}

\begin{proof}
Suppose $E \in W_{i,\Phi}$ where $0\leq i \leq 1$.  Then $\Phi E \cong \wh{E} [-i]$ for some $\wh{E} \in \Ac_X$.  By Property A2, we have $E[-1] \cong (\wh{\Phi} \wh{E})[-i]$, i.e.\ $\wh{\Phi} \wh{E} \cong E [-(1-i)]$, and the lemma follows.
\end{proof}

\paragraph[Torsion classes in $\Ac_X$] We collect here  basic examples of torsion classes in $\Ac_X$.

\begin{itemize}
\item \textbf{Property TC1.} For any   variety $Y$ and any nonnegative integer $d$, the category $\Ac_Y^{\leq d}$ is a Serre subcategory of $\Ac_Y$.
\item \textbf{Property TC2.} For any morphism of   varieties $p : Y \to Z$ and any nonnegative integer $e$, the category $\Ac(p)_{\leq e}$ is a Serre subcategory of $\Ac_Y$.
\item \textbf{Property TC3.} $(W_{0,\Phi}, W_{1,\Phi})$ is a torsion pair in $\Ac_X$.  (In particular, $W_{0,\Phi}$ is a torsion class in $\Ac_X$.)
\end{itemize}
Property TC3 follows from \cite[Lemma 9.2]{BMef} and properties of the heart of a t-structure.  Since the category of coherent sheaves on any algebraic variety $Y$ is a noetherian abelian category, any Serre subcategory of $\Ac_Y = \Coh (Y)$ is a torsion class in $\Ac_Y$ by \cite[Lemma 1.1.3]{Pol}.

\paragraph[Chern characters and the product threefold]\label{para:Xproduct}  In the special case where the elliptic threefold $X$ is the product $C \times B$ of a smooth elliptic curve and a K3 surface $B$ of Picard rank 1, with the second projection $\pi : X\to B$ as the fibration map, we have a matrix notation for the Chern characters of objects in $D^b(\Ac_X)$.  In this case,   the Chern character $\ch(E)$ of any object $E \in D^b(X)$ can be represented by a $2 \times 3$ matrix of integers by the second author's joint work with Zhang \cite[2.1]{Lo14}, and we have the following Properties on Chern characters.   The labels of these  Properties  end with a lower case `p' to indicate that, as stated, they are specific to the product threefold case.  We expect there to be slight modifications of these Properties (taking into account twists by $B$-fields) that hold for a general Weierstra{\ss} elliptic threefold (e.g.\ see \cite{Lo15}).

\begin{itemize}
\item \textbf{Property CH0p.} For every $E \in D^b(\Ac_X)$, there is an associated $2 \times 3$ matrix
    \begin{equation}\label{eq:chE}
    \ch(E)=(\alpha_{ij}) = \begin{pmatrix} \al_{00} & \al_{01} & \al_{02} \\ \al_{10} & \al_{11} & \al_{12} \end{pmatrix}
     \end{equation}
     where $\alpha_{ij} \in \mathbb{Z}$ for $0\leq i \leq 1, 0 \leq j \leq 2$.
\end{itemize}

The six entries in the matrix correspond to the six generators of the algebraic cohomology ring  of $X$ over $\mathbb{Z}$.  For instance, in \eqref{eq:chE}, the entry $\alpha_{00}$ represents the rank of $E$, while $\alpha_{10}$ represents the fiber degree of $E$, i.e.\ $f\ch_1(E)$ where $f$ denotes the fiber class of the fibration $\pi$.

\begin{itemize}
\item \textbf{Property CH1p.} For any $E \in \Ac_X$, 
\[
  \codimension E =  \min{\{ i+j : \al_{ij} \neq 0, 0\leq i \leq 1, 0 \leq j \leq 2\}}.
\]
\item \textbf{Property CH2p.} For any $E \in \Ac_X$, 
\[
  \dimension (\pi (\supp E)) = \max{\{2-j : \alpha_{1j} \neq 0\}}.
\]
\item \textbf{Property CH3p.} If $E$ is a nonzero object in $\Ac_X$, then $\sum_{i+j=\codimension E} \alpha_{ij} > 0$.
\item \textbf{Property A4p.} For any $E \in D^b(\Ac_X)$, if $ \ch(E) = \begin{pmatrix} \al_{00} & \al_{01} & \al_{02} \\
\al_{10} & \al_{11} & \al_{12} \end{pmatrix}$ then
\begin{align}\label{eq:PhichEformula}
\ch(\Phi E) &= \begin{pmatrix} \al_{10} & \al_{11} & \al_{12} \\
-\al_{00} & -\al_{01} & -\al_{02} \end{pmatrix}, \\
\ch (E[1]) &= \begin{pmatrix} -\al_{00} & -\al_{01} & -\al_{02} \\
-\al_{10} & -\al_{11} & -\al_{12} \end{pmatrix}.
\end{align}
\end{itemize}

Properties CH1p and CH3p follow from  \cite[Proposition 5.13]{LZ2}.  Property A4p is \cite[Proposition 4.5]{LZ2}.

\begin{lem}\label{Daniel}
$\Ac^{\leq 2}_X \cap W_{1,\Phi} \subseteq \Ac (\pi)_{\leq 1}$.
\end{lem}

\begin{proof}
Take any $E \in \Ac^{\leq 2}_X \cap W_{1,\Phi}$.  Suppose $\dimension E \leq 1$; then Property D2 gives $E \in \Ac (\pi)_{\leq 1}$.  So let us suppose $\dimension E = 2$ from now on.

That $\dimension E = 2$ implies $\codimension E = 1$.  By Property CH1p, we have  $\alpha_{00}=0$ and  $\al_{10} \geq 0$. On the other hand, by Property A4p we have $ \ch(\Phi E)[1] = \begin{pmatrix} -\al_{10} & -\al_{11} & -\al_{12} \\
\al_{00} & \al_{01} & \al_{02} \end{pmatrix}$.  That $E$ is $\Phi$-WIT$_1$ means $\Phi E [1] \in \Ac_X$, and so $-\alpha_{10} \geq 0$ by Properties CH1p and CH3p.  Overall, we have $\alpha_{10}=0$.  Property CH2p now gives $\dimension (\pi (\supp E)) \leq 1$, i.e.\ $E \in \Ac (\pi)_{\leq 1}$.
\end{proof}

\section{Torsion classes in the category of coherent sheaves}\label{sec:computations}

The results in this section that are proved without using Properties CH1p, CH2p, CH3p and A4p hold for a general Weierstra{\ss} threefold $X$, while those that are proved using these four properties hold only for the product threefold  $X = C \times B$  as in \ref{para:Xproduct}.

Some of the results  in this section have already appeared in \cite{Lo14, Lo15}.  All the proofs in this section, however, rely only on the Properties listed in  Section \ref{se:propertieslist}, Properties C0 and C1 below, and  the preliminary notions in Section \ref{sec:prelim}.

We begin by introducing  subcategories $C_{ij}$ of $\Ac_X$ that will form the building blocks of various torsion  classes in $\Ac_X$.  The categories $C_{ij}$ are defined for pairs  $(i,j)$  taken from the collection
\begin{equation}\label{eq:ijindex}
\scalebox{0.9}{\xymatrix @-2pc{
  (0,0) & (1,0) & (2,0) & (3,0) & (4,0) & (5,0) \\
  & (1,1) & & (3,1) & & (5,1) \\
  & (1,2) & & (3,2) & & (5,2)
}
}
\end{equation}
and their definitions are:

\begin{align*}
C_{00} &= \Ac_X^{\leq 0} \\
C_{10} &= \{ E \in \Ac(\pi)_0 \cap W_{0,\Phi}: \Hom (C_{00},E)=0\} \\
C_{11} &=  \{ E \in \Ac(\pi)_0 \cap W_{1,\Phi} : \dimension \wh{E} =0\} \\
C_{12} &= \{ E \in \Ac(\pi)_0 \cap W_{1,\Phi} : \dimension \wh{E} = 1, \Hom (C_{11},E)=0\} \\
  C_{20} &=  \Ac^{\leq 1}_h 
\end{align*}

The category $C_{00}$ is precisely the category of coherent sheaves on $X$ supported at points.  By the classification theorem of semistable sheaves on integral genus one curves \cite[Proposition 6.38]{FMNT}, the categories $C_{10}, C_{11}, C_{12}$ are exactly the extension closures of: semistable fiber sheaves all of whose Harder-Narasimhan (HN) factors have strictly positive slopes; semistable fiber sheaves all of whose HN factors have slope 0; semistable fiber sheaves all of whose HN factors have strictly negative slopes.  The category $C_{20}$ is the category of pure 1-dimensional sheaves $F$ such that each irreducible component of $\supp F$  is `horizontal'.  In the diagram notation of \cite{Lo14,Lo15}, the categories $C_{00}, C_{10}, C_{11}, C_{12}, C_{20}$ are $\scalea{\gyoung(;;;,;;;+)},\scalea{\gyoung(;;;+,;;;+)},\scalea{\gyoung(;;;+,;;;0)},\scalea{\gyoung(;;;+,;;;-)},\scalea{\gyoung(;;;*,;;+;*)}$, respectively.  

\begin{align*}
C_{30} &= \Ac^2(\pi)_1 \cap W_{0,\Phi} \\
 C_{31} &= \{ E \in \Phi ( C_{20}) : \dimension E=2\}\\
C_{32} &= \{ E \in \Ac^2(\pi)_1 \cap W_{1,\Phi} : \dimension \wh{E}=2\} \\
C_{40} &= \{ E \in \Ac^2(\pi)_2 \cap W_{0,\Phi} : \dimension \wh{E}=3\} \\
C_{50} &= \Ac^3(\pi)_2 \cap W_{0,\Phi} \\
C_{51} &= \{ E \in \Ac^3(\pi)_2 \cap W_{1,\Phi} : \dimension \wh{E} =2\} \\
C_{52} &= \{ E \in \Ac^3 (\pi)_2 \cap W_{1,\Phi} : \dimension \wh{E}=3\}
\end{align*}

The categories  $C_{30}, C_{31}, C_{32}, C_{40}, C_{50}, C_{51}, C_{52}$ are   $\scalea{\gyoung(;;+;*,;;+;*)},\scalea{\gyoung(;;+;*,;;0;*)},\scalea{\gyoung(;;+;*,;;-;*)},\scalea{\gyoung(;;*;*,;+;*;*)},\scalea{\gyoung(;+;*;*,;+;*;*)},\scalea{\gyoung(;+;*;*,;0;*;*)},\scalea{\gyoung(;+;*;*,;-;*;*)}$, respectively, in the papers \cite{Lo14, Lo15}.  Note that the categories $C_{ij}$ can also be defined with $\Phi$ replaced with $\wh{\Phi}$; we use the same notation to denote such categories because of the symmetry in $\Phi$ and $\wh{\Phi}$.  There should be no risk of confusion. 

Properties of the categories $C_{ij}$ include:

\begin{itemize}
\item \textbf{Property C0.} $C_{00} \subset W_{0,\Phi}$.
\item \textbf{Property C1.} $C_{20} \subset W_{0,\Phi}$.
\end{itemize}

Property C0 follows from the fact that the kernel of the Fourier-Mukai functor $\Phi$ is a universal sheaf for a moduli problem on $X$ \cite{BMef}.  Property C1 is \cite[Lemma 3.6]{Lo11}.  Note that we have  $C_{31} \subset W_{1,\Phi}$ by Properties C1 and Lemma \ref{lem:W01transform}.

\begin{lem}\label{lem:C00isTC}
For any $0 \leq m \leq 2$, the category $\Ac^{\leq m}_X$ is a Serre subcategory of  $\Ac_X$, as is the category $\Ac (\pi)_{\leq m}$.
\end{lem}
\begin{proof}
The first assertion follows from Property D1, while the second assertion follows from Property D3.
\end{proof}

\begin{lem}\label{lem:Ac_0isTC}
$\Ac (\pi)_0$ is a Serre subcategory of  $\Ac_X$.
\end{lem}

\begin{proof}
This follows from Property D3.
\end{proof}

\begin{lem}\label{lem:C00inAcpi0}
$C_{00} \subset \Ac (\pi)_0$.
\end{lem}

\begin{proof}
For any $E \in C_{00}$, we have $\dimension (\pi (\supp E)) \leq 0$ by D2, and so $E \in \Ac (\pi)_0$. 
\end{proof}

\begin{lem}\label{lem:basic1}
If $E \in C_{10}$, then any $\Ac_X$-quotient $E'$ of $E$ lies in $\langle C_{00}, C_{10}\rangle$.
\end{lem}

\begin{proof}
Let $E, E'$ be as described.  Since $C_{00} = \Ac_X^{\leq 0}$ is a torsion class in $\Ac_X$ and $C_{00}$ is contained in the abelian subcategory $\Ac (\pi)_0$ of $\Ac_X$, it follows that $C_{00}$ is a torsion class in $\Ac (\pi)_0$.  Hence we have a short exact sequence in $\Ac (\pi)_0$
\[
  0 \to E'_0 \to E' \to E'_1 \to 0
\]
where $E'_0 \in C_{00}$ while $E'_1$ satisfies $\Hom (C_{00},E'_1)=0$.  Since $\Ac (\pi)_0$ and $W_{0,\Phi}$ are both torsion classes in $\Ac_X$, $E_1'$ must also lie in $\Ac (\pi)_0 \cap W_{0,\Phi}$.  Hence $E'_1 \in C_{10}$.
\end{proof}

\begin{lem}\label{lem:C11reformulation}
 $C_{11} = \Phi C_{00}$.
\end{lem}

\begin{proof}
Since $C_{00} \subset W_{0,\Phi}$ by C0, we have $\Phi C_{00} \subset W_{1,\wh{\Phi}}$ by Lemma \ref{lem:W01transform}.   Since $C_{00} \subset \Ac (\pi)_0$ by Lemma \ref{lem:C00inAcpi0}, we have $\Phi C_{00} \subset \Ac (\pi)_0$ by A3.    Also, for any $E \in C_{00}$, we have $\wh{E} \cong \Phi E$ is $\wh{\Phi}$-WIT$_1$ from above, and so $\wh{\wh{E}} \cong  \wh{\Phi} (\Phi E) [1] \cong E$, i.e.\ $\dimension \wh{\wh{E}}=0$.  We have now shown $\Phi C_{00} \subseteq C_{11}$.  The other inclusion follows from Property C0 and Lemma \ref{lem:W01transform}.
\end{proof}

\begin{lem}\label{lem:C12reformulation}
$C_{12} = \Phi C_{10}$.
\end{lem}

\begin{proof}
By the same argument as  in the proof of Lemma \ref{lem:C11reformulation}, we obtain $\Phi C_{10} \subset \Ac (\pi)_0 \cap W_{1,\wh{\Phi}}$.  Also, if $E \in \Phi C_{10}$ is a nonzero object, then $\wh{E} \in C_{10}$, and by the definition of $C_{10}$, the object $\wh{E}$ cannot be supported in dimension 0, forcing $\dimension \wh{E}=1$.  Now, for any $A \in C_{10}$ and $B \in C_{11}$, we have $\Hom (B,\Phi A) \cong \Hom (\wh{B},A)$ since $\wh{\Phi}$ is an equivalence and by Property A2.  Since $\wh{B} \in C_{00}$, we must have $\Hom (\wh{B},A)=0$ from the definition of $C_{10}$.  Hence $\Hom (B,\Phi A)=0$, and we have shown $\Phi C_{10} \subseteq C_{12}$.

To see the other inclusion,  take any $E \in C_{12}$.  The same argument as above shows that $\wh{E} \in \Ac (\pi)_0 \cap W_{0,\wh{\Phi}}$; we also have  $\dimension \wh{E}=1$ from the definition of $C_{12}$.  It remains to show $\Hom (C_{00},\wh{E})=0$.  Suppose we have a  morphism $\alpha : A \to \wh{E}$ for some $A \in C_{00}$.  Since $A$ and $\wh{E}$ are both $\wh{\Phi}$-WIT$_0$, the functor $\wh{\Phi}$ takes $\alpha$ to the  morphism $\wh{\Phi} \alpha : \wh{A} \to E$.  Now we have  $\wh{A} \in C_{11}$ by Lemma \ref{lem:C11reformulation}, and so $\wh{\Phi}\alpha$ and hence $\alpha$ must be zero from the definition of $C_{12}$.  This completes the proof of the lemma.
\end{proof}

\begin{lem}\label{lem:C00C20isTC}
$\langle C_{00}, C_{20}\rangle$ is a torsion class in $\Ac_X$.
\end{lem}

\begin{proof}
We already know $C_{00}$ is a torsion class from Lemma \ref{lem:C00isTC}.  Therefore,  it suffices to take an arbitrary $E \in C_{20}$ and any $\Ac_X$-surjection $E \twoheadrightarrow E''$, and show that  $E''$ lies in $\langle C_{00}, C_{20}\rangle$.  If $\dimension E'' = 0$, then $E'' \in C_{00}$; so let us assume $\dimension E'' = 1$.

Since $E \in C_{20}=\Ac^{\leq 1}_h$, all the 1-dimensional irreducible components of $\supp (E)$ are of type (b) in Z2.  Since $\supp (E'') \subseteq \supp (E)$ by Property D1, the 1-dimensional irreducible components of $\supp (E'')$ are also of type (b) in Z2, i.e.\ $E'' \in \Ac^{\leq 1}_h=C_{20}$. 
\end{proof}


\begin{defn}\label{Jazz}
If a pair of integers $(i,j)$ is not part of the collection \eqref{eq:ijindex}, we define $C_{ij}$ to be empty.  For $n=0,2,4$, we define 
\[
  \Tc_{n0} = \langle C_{ij} : i \leq n, 0 \leq j \leq 2 \rangle.
\]
For $n= 1, 3, 5$, we define 
\begin{align*}
  \Tc_{n0} &= \langle \Tc_{(n-1),0}, C_{n0} \rangle, \\
  \Tc_{n1} &= \langle \Tc_{n0}, C_{n1} \rangle, \\
  \Tc_{n2} &= \langle \Tc_{n1}, C_{n2} \rangle.
\end{align*}
\end{defn}

\begin{eg}
\[
\Tc_{20} = \biggg< \begin{matrix} C_{00} & C_{10} & C_{20} \\
& C_{11} & \\
& C_{12} & \end{matrix}
\biggg> \text{\quad and \quad} \Tc_{31} = 
\biggg< \begin{matrix} C_{00} & C_{10} & C_{20} & C_{30}\\
& C_{11} &&  C_{31} \\
& C_{12} &  & \end{matrix}
\biggg>.
\]
\end{eg}

\begin{defn}\label{Jonathan}
For $2 \leq i \leq 5$, we  define $\mathcal{F}_i=\langle C_{00}, C_{10},\ldots, C_{i0} \rangle$.
\end{defn}

\begin{lem}\label{Gabriel}
Suppose we have $\Ac_X$-short exact sequences 
\begin{gather*}
 0 \to K \to A \to Q \to 0 \\
  0 \to Q_0 \to Q \to Q_1 \to 0 
\end{gather*}
where  $Q_i \in W_{i,\Phi}$ for $i=0,1$.  If $A$ is $\Phi$-WIT$_1$, then  there exists an $\Ac_X$-surjection  $\wh{A} \twoheadrightarrow \wh{Q}_1$.
\end{lem}

\begin{lem}\label{Bill}
Suppose $\Tc, \mathcal{C}$ are subcategories of $\Ac_X$, with $\Tc$ being a torsion class  in $\Ac_X$.  Suppose that for every $A \in \mathcal{C}$ and every $\Ac_X$-quotient $A \twoheadrightarrow A'$, we have $A' \in \langle \Tc, \mathcal{C}\rangle$.  Then $\langle \Tc, \mathcal{C}\rangle$ is a torsion class in $\Ac_X$.
\end{lem}

\begin{lem} \label{Eddie}
Suppose $1 \leq m \leq 3$ is an integer and  $E\in \mathcal{A}^m(\pi)_{m-1} \cap W_{0,\Phi}$. For for any  $\Ac_X$-quotient $E \twoheadrightarrow E'$ where  $\dim(E')=m$, we have  $E' \in \mathcal{A}^m(\pi)_{m-1} \cap W_{0,\Phi}$
\end{lem}

\begin{proof}
Let $E,E'$ be as described. From Property TC3, we know $W_{0,\Phi}$ is a torsion class in $\Ac_X$, and so  $E' \in W_{0,\Phi}$.  Since $\dim(E')= m$, we know $\dimension (\pi (\supp E'))$ is either $m-1$ or $m$ by Property D2.  Property D3, however, gives that $\dimension (\pi (\supp  E')) \leq m-1$, and so $\dimension (\pi (\supp E')) = m-1$ and we are done.
\end{proof}

\begin{lem}\label{Jade}
$\mathcal{A}(\pi)_0 \cap W_{0,\Phi} = \mathcal{T}_{10}$, and it is a torsion class in $\Ac_X$.
\end{lem}

\begin{proof}
We know $\Ac (\pi)_0$ and $W_{0,\Phi}$ are both torsion classes in $\Ac_X$ from Lemma \ref{lem:Ac_0isTC} and Property TC3.  That the intersection of two torsion classes is again a torsion class follows immediately from the definition of a torsion class.  

Recall that $\mathcal{T}_{10}= \langle C_{00},C_{10}\rangle$.  We have $C_{10} \subseteq \Ac (\pi)_0 \cap W_{0,\Phi}$ by definition, and $C_{00} \subseteq \Ac (\pi)_0 \cap W_{0,\Phi}$ by Properties D2 and C0.

To see the other inclusion, take any $E \in \mathcal{A}(\pi)_0 \cap W_{0,\Phi}$. Since $\Ac (\pi)_0$ is a Serre subcategory of $\Ac_X$, that  $C_{00}$ is a torsion class in $\Ac_X$ (Lemma \ref{lem:C00isTC}) implies $C_{00}$ is also a torsion class in the abelian category $\Ac (\pi)_0$; let $\Fc$ denote the corresponding torsion-free class in $\Ac (\pi)_0$.  Then we have an $\Ac (\pi)_0$-short exact sequence
\[
0 \to E' \to E \to E'' \to 0
\]
where $E' \in C_{00}$ and $E'' \in \Fc$.  That $\Ac (\pi)_0 \cap W_{0,\Phi}$ is a torsion class in $\Ac_X$ (hence in $\Ac (\pi)_0$) implies $E'' \in \Ac (\pi)_0 \cap W_{0,\Phi}$.  Since $\Hom (C_{00},E'')=0$ by construction, we have $E'' \in C_{10}$, giving us $E \in \Tc_{10}$.  This completes the proof of the lemma.
\end{proof}

\begin{lem}\label{Lucas}
$\mathcal{T}_{11}$ is a torsion class in $\Ac_X^{\leq1}$.
\end{lem}

\begin{proof} 
Recall that $\mathcal{T}_{11}= \langle \mathcal{T}_{10}, C_{11}\rangle$.  Take any  $E \in C_{11}$ and any $\Ac_X$-quotient $E \twoheadrightarrow E'$.  By Lemma \ref{Jade}, $\Ac (\pi)_0 \cap W_{0,\Phi}$ is a torsion class in $\Ac_X$, hence in $\Ac^{\leq 1}_X$; let $\Fc$ be the corresponding torsion free class in $\Ac^{\leq 1}_X$.  Then we have an $\Ac^{\leq 1}_X$-short exact sequence
\[
 0 \to A_0 \to E' \to A_1 \to 0
\]
where $A_0 \in \Ac(\pi)_0 \cap W_{0,\Phi}=\Tc_{10}$ and $A_1 \in \Fc$.  By Lemmas \ref{Jade} and \ref{Bill}, it suffices for us to show  $A_1 \in \Tc_{11}$.  That $\Ac (\pi)_0$ is a torsion class in $\Ac_X$ (Lemma \ref{lem:Ac_0isTC}) implies $A_1 \in \Ac (\pi)_0$.  On the other hand, if $A_{10}$ denotes the $\Phi$-WIT$_0$ part of $A_{10}$ in $\Ac_X$, then $A_{10} \in \Ac (\pi)_0$ by Lemma \ref{lem:Ac_0isTC} again, i.e.\ $A_{10} \in \Ac (\pi)_0 \cap W_{0,\Phi}$, and then $A_{10}=0$ by construction of $A_1$.  That is, we have $A_1 \in \Ac (\pi)_0 \cap W_{1,\Phi}$.  By Lemma \ref{Gabriel}, we now have an $\Ac_X$-surjection $\wh{E} \to \wh{A_1}$.  Since $\dimension \wh{E}=0$ from the definition of $C_{11}$, we have $\dimension \wh{A_1}=0$ by Lemma \ref{lem:C00isTC}.  This shows $A_1 \in C_{11}$, and we are done.
\end{proof}

\begin{lem}\label{Sarah}
$\Ac (\pi)_0 = \Tc_{12}$ is a Serre subcategory of $\Ac_X$, and hence a torsion class in $\Ac_X$.
\end{lem}

\begin{proof}
We know that $\Ac (\pi)_0$ is a Serre subcategory of $\Ac_X$ from Lemma \ref{lem:Ac_0isTC}; since $\Ac_X$ is a noetherian abelian category, it follows that $\Ac (\pi)_0$ is a torsion class in $\Ac_X$.  Hence it suffices to show the equality $\Ac (\pi)_0 = \Tc_{12}$.

That $\Tc_{12} \subseteq \Ac (\pi)_0$ is clear from Lemma \ref{lem:C00inAcpi0} and the construction of $\Tc_{12}$.  To show the other inclusion, take any $E \in \Ac (\pi)_0$.  Since $E \in \Ac_X^{\leq 1}$ and $\Tc_{11}$ is a torsion class in $\Ac_X^{\leq 1}$ by Lemma \ref{Lucas}, there is an $\Ac^{\leq 1}_X$-short exact sequence
\[
0 \to E' \to E \to E'' \to 0
\]
where $E' \in\Tc_{11} \subset \Tc_{12}$ while $\Hom (\Tc_{11}, E'')=0$. 

Since $\Ac (\pi)_0$ is a Serre subcategory of $\Ac_X$, the $\Phi$-WIT$_0$ component of $E''$ must lie in $\Ac (\pi)_0$; however, we have $\Hom (\Tc_{10},E'')=0$, and so by Lemma \ref{Jade} we must have  $E'' \in W_{1,\Phi}$.  That is, $E'' \in \Ac(\pi)_0 \cap W_{1,\Phi}$.  Since $C_{11} \subset \Tc_{11}$, we also have $\Hom (C_{11}, E'')=0$.  On the other hand,  Property A3 gives  $\wh{E''} \in \Ac (\pi)_0$, and so $\dimension \wh{E''} \leq 1$.  If $\dimension \wh{E''}=1$ then $E'' \in C_{12}$; if $\dimension \wh{E''}=0$ then $E'' \in C_{11}$.  Hence $E'' \in \Tc_{12}$, implying $E \in \Tc_{12}$, and we are done.
\end{proof}

\begin{lem} \label{Max}
$\mathcal{A}^{\leq 1}_X \cap W_{0,\Phi} =\mathcal{F}_2$, and it is a torsion class in $\Ac_X$.
\end{lem}

\begin{proof}
Since $\Ac^{\leq 1}_X$ and $W_{0,\Phi}$ are both torsion classes in $\Ac_X$ (by Lemma \ref{lem:C00isTC} and Property TC3), their intersection is also a torsion class in $\Ac_X$.  

The inclusion $\Fc_2 \subseteq \Ac^{\leq 1}_X \cap W_{0,\Phi}$ follows from Properties C0, C1 and the definition of $C_{10}$.  To see the inclusion $\Ac^{\leq 1}_X \cap W_{0,\Phi} \subseteq \Fc_2$, take any $E \in \Ac^{\leq 1}_X \cap W_{0,\Phi}$.  From Lemma \ref{lem:C00C20isTC}, the extension closure $\langle C_{00}, C_{20}\rangle$ is a torsion class in $\Ac^{\leq 1}_X$; let $\Fc$ denote the corresponding torsion-free class in $\Ac^{\leq 1}_X$.  Then there exists an $\Ac^{\leq 1}_X$-short exact sequence
\[
0 \to E' \to E \to E'' \to 0
\]
where $E' \in \langle C_{00}, C_{20}\rangle$ and $E'' \in \Fc$.  Since $E' \in \Fc_2$, it suffices to show $E'' \in \Fc_2$ when $E''$ is nonzero.

Since $E \in W_{0,\Phi}$, we have $E'' \in W_{0,\Phi}$ by Property TC3.  By the definition of $\Fc$, we have the vanishing $\Hom (C_{00},E'')=0$, and so $\dimension E''$ must be 1. By repeatedly applying Property Z1, and using Property Z2, we can construct an $\Ac_X$-short exact sequence
\[
0 \to A_h \to E'' \to A_v \to 0
\]
where $A_h$ lies in either $C_{00}$ or $\Ac^{\leq 1}_h = C_{20}$, and all the irreducible components of $\supp (A_v)$ are contained in fibers of $\pi$ (i.e.\ $A_v  \in \Ac(\pi)_0$).  Hence $A_h \in \Fc_2$ while $A_v \in \Ac (\pi)_0 \cap W_{0,\Phi} = \Tc_{10} \subset \Fc_2$  (by Lemma \ref{Jade}), and we are done.
\end{proof}

\begin{lem}\label{Tyler}
$\mathcal{T}_{20}$ is a torsion class in $\Ac_X$.
\end{lem}

\begin{proof}
We have $\mathcal{T}_{20}= \langle \mathcal{T}_{12}, C_{20}\rangle$ by definition.  For any  $E \in C_{20}$ and any $\mathcal{A}_X$-quotient $E'$ of $E$, we have $E \in \Fc_2$ and hence $E' \in \Fc_2 \subset \Tc_{20}$ by Lemma \ref{Max}.  Then by Lemmas \ref{Sarah} and \ref{Bill}, we conclude $\mathcal{T}_{20}$ is a torsion class in $\Ac_X$.
\end{proof}

\begin{lem}\label{Jamie}
$\mathcal{F}_3$ is a torsion class in $\Ac_X$.
\end{lem}

\begin{proof}
We have $\mathcal{F}_{3}= \langle \mathcal{F}_{2}, C_{30}\rangle$ by definition.  Take any $E \in C_{30}$ and any $\mathcal{A}_X$-surjection $E \twoheadrightarrow E'$.  By Lemmas \ref{Max} and \ref{Bill}, it suffices to show $E' \in \Fc_3$.

By Property TC3 we have  $E' \in W_{0,\Phi}$, while by Property D1 $\dim(E') \leq 2$. If $\dim(E') \leq 1$, then $E' \in \mathcal{F}_2$ by Lemma \ref{Max}; if $\dim(E') = 2$, then by Lemma \ref{Eddie}, we have  $E' \in C_{30}$. This completes the proof.
\end{proof}

\begin{lem}\label{Zuly}
$\mathcal{T}_{30}$ is a torsion class in $\Ac_X$.
\end{lem}

\begin{proof}
We have $\mathcal{T}_{30}= \langle \mathcal{T}_{20}, C_{30}\rangle$ by definition.  Take any $E \in C_{30}$ and any $\Ac_X$-surjection $E \twoheadrightarrow E'$.  By Lemmas \ref{Tyler} and \ref{Bill}, it suffices to show $E' \in \Tc_{30}$.  Since $C_{30} \subset \Fc_{3}$, by Lemma \ref{Jamie} we have $E' \in \Fc_3 \subset \Tc_{30}$, and we are done.
\end{proof}

\begin{rem}\label{rem:C31-note1}
We have $C_{31} \subset \Ac (\pi)_{\leq 1}$.  To see this, note that $C_{20} = \Ac^{\leq 1}_h$ is contained in $\Ac (\pi)_{\leq 1}$ by Property D2.  Then by Properties C1 and A3, we have $C_{31} \subseteq \Phi (C_{20}) \subset \Ac (\pi)_{\leq 1}$.
\end{rem}

\begin{lem}\label{Gaby}
$\mathcal{T}_{31}$ is a torsion class in $\Ac_X$.
\end{lem}

\begin{proof}
We have $\mathcal{T}_{31}= \langle \mathcal{T}_{30}, C_{31}\rangle$ by definition.  Take any $E \in C_{31}$ and any $\Ac_X$-surjection $E \twoheadrightarrow E'$.  By Lemmas \ref{Zuly} and \ref{Bill}, it suffices for us to show $E' \in \Tc_{31}$.  By Property TC3, there is an $\Ac_X$-short exact sequence
\[
0 \to A_0 \to E' \to A_1 \to 0
\]
where $A_i \in W_{i,\Phi}$. Note that $\dimension A_0, \dimension A_1 \leq 2$ by Property D1. 

If $\dimension A_0 \leq 1$, then $A_0 \in \Fc_2$ by Lemma \ref{Max}, and so $A_0 \in \Tc_{31}$.  Now suppose  $\dimension A_0=2$; since $E \in \Ac (\pi)_{\leq 1}$ by Remark \ref{rem:C31-note1}, it follows that $A_0 \in \Ac(\pi)_{\leq 1}$ by Lemma \ref{lem:C00isTC}, i.e.\ $A_0 \in \Ac (\pi)_{\leq 1} \cap W_{0,\Phi}$.  Note that $\dimension (\pi (\supp A_0))$ must be exactly 1 by Property D2, and so $A_0 \in C_{30} \subset \Tc_{31}$.

As for $A_1$, we have $A_1 \in \Tc_{31}$ if $\dimension A_1 \leq 1$ as above, so let us assume $\dimension A_1=2$.  Note that $C_{31}$ is contained in $W_{1,\Phi}$ by Property C1 and Lemma \ref{lem:W01transform}.  Applying Lemma \ref{Gabriel} to the composite $\Ac_X$-surjection $E\twoheadrightarrow E' \twoheadrightarrow A_1$ then gives $\wh{A_1} \in C_{20}$, which shows $A_1 \in C_{31} \subset \Tc_{31}$.  

Overall, we have $E' \in \Tc_{31}$, completing the proof.
\end{proof}

\begin{lem}\label{lem:T20isAleq1}
$\mathcal{T}_{20} = \Ac_X^{\leq 1}$.
\end{lem}

\begin{proof}
The inclusion $\mathcal{T}_{20} \subseteq \Ac_X^{\leq 1}$ is clear.  To see the other inclusion, take any $E \in \Ac_X^{\leq 1}$.  By Properties Z1 and Z2, there is an $\Ac_X^{\leq 1}$-short exact sequence 
\[
0 \to E_h \to E \to E_v \to 0
\]
where $E_h$ lies in $C_{00}$ (if $\dimension E_h=0$) or $C_{20}$ (if $\dimension E_h=1$), and $E_v$ lies in  $\Ac (\pi)_0$.  Since $\Ac (\pi)_0 = \Tc_{12}$ by Lemma \ref{Sarah}, we see that $E$ lies in $\Tc_{20}$.
\end{proof}

\begin{lem}\label{Erik}
Suppose $Q \in W_{1,\Phi}$ satisfies $\wh{Q}\in \mathcal{A}^{\leq1}$.  Then  $Q\in \mathcal{T}_{31}$. 
\end{lem}

\begin{proof}
As in the proof of Lemma \ref{lem:T20isAleq1}, there is an $\Ac^{\leq 1}_X$-short exact sequence
\begin{equation}\label{eq:Erik-1}
0 \to A_h \to \wh{Q} \to A_v \to 0
\end{equation}
where $A_h \in C_{00} \cup C_{20}$ and $A_v \in\Ac(\pi)_0$.  Note that $\wh{Q} \in W_{0,\wh{\Phi}}$ by Lemma \ref{lem:W01transform}, and so $A_v \in W_{0,\wh{\Phi}}$ by Property TC3. By Property A3 and Lemma \ref{lem:W01transform}, we now have  $\wh{A_v} \in \Ac (\pi)_0$, and so $\wh{A_v}\in \Tc_{12} \subset \Tc_{31}$ by Lemma \ref{Sarah}.

On the other hand, note that $A_h \in W_{0,\wh{\Phi}}$ by Properties C0 and C1.  If $A_h \in C_{00}$ then $\wh{A_h} \in C_{11} \subset \Tc_{31}$ by Lemma \ref{lem:C11reformulation}.  If $A_h \in C_{20}$, then there are two cases for $\wh{A_h}$: (i) if $\dimension \wh{A_h}=2$, then $\wh{A_h} \in C_{31}$ by the definition of $C_{31}$, in which case $\wh{A_h} \in \Tc_{31}$; (ii) if $\dimension \wh{A_h} \leq 1$, then $\wh{A_h} \in \Tc_{20}$ by Lemma \ref{lem:T20isAleq1}, and we still have $\wh{A_h} \in \Tc_{31}$.

Thus both $\wh{A_h}, \wh{A_v}$ lie in $\Tc_{31}$.  Since $\Phi$ takes the short exact sequence \eqref{eq:Erik-1} to the $\Ac_X$-short exact sequence
\[
0 \to \wh{A_h} \to Q \to \wh{A_v} \to 0,
\]
we see that $Q$ itself lies in $\Tc_{31}$.
\end{proof}

\begin{lem}\label{Manuel}
$\mathcal{T}_{32}$ is a torsion class in $\Ac_X$.
\end{lem}

\begin{proof} 
We have $\mathcal{T}_{32}= \langle \mathcal{T}_{31}, C_{32}\rangle$ by definition.  Take any $E \in C_{32}$ and any $\Ac_X$-surjection $E \twoheadrightarrow E'$.  By Lemmas \ref{Gaby} and \ref{Bill}, it suffices for us to show $E' \in \Tc_{32}$.  By Property TC3, there is an $\Ac_X$-short exact sequence 
\[
 0 \to A_0 \to E' \to A_1 \to 0
 \]
where $A_i \in W_{i,\Phi}$.   Property D1 implies  $\dim(A_0),\dim(A_1)\leq 2$.  Since $\Ac_X^{\leq 1} = \Tc_{20} \subset \Tc_{32}$ by Lemma \ref{lem:T20isAleq1}, we can further assume $\dimension A_0 = \dimension A_1 = 2$.

That $E' \in C_{32}$ means $E' \in \Ac(\pi)_{\leq 1}$, and so by Properties D2 and D3,  we  have $\dimension (\pi (\supp A_0)) = \dimension (\pi (\supp A_1))=1$.  Hence $A_0 \in C_{30} \subset \Tc_{32}$.

On the other hand, we have  $A_1 \in \mathcal{A}^2(\pi)_1\cap \Wi$. By Properties A3 and D2 we know  $\dim(\wh{A_1})$ is either 1 or 2.  If $\dim(\wh{A_1}) = 1$, then $A_1 \in \Tc_{31} \subset \Tc_{32}$ by Lemma \ref{Erik};  if $\dimension \wh{A_1}=2$, then $A_1 \in C_{32} \subset \Tc_{32}$.  Thus $E' \in \Tc_{32}$, and we are done.
\end{proof}

\begin{lem}\label{Esteban}
$\Ac_X^{\leq 2} = \mathcal{T}_{40}$.
\end{lem}

\begin{proof}
That $\mathcal{T}_{40} \subseteq \Ac^{\leq 2}_X$ is clear.  To show the other inclusion,  take any $E \in \Ac^{\leq 1}_X$.  Since $\Ac^{\leq 2}_X$ is a Serre subcategory of $\Ac_X$ by Lemma \ref{lem:C00isTC}, Property TC3 implies that there is an $\Ac_X^{\leq 2}$-short exact sequence
\[
0 \to E_0 \to E \to E_1 \to 0
\]
where $E_i \in W_{i,\Phi}$.  Given Lemma \ref{lem:T20isAleq1}, it suffices to assume $\dimension E_0 = \dimension E_1 = 2$.

Lemma \ref{Daniel} and Property D2 together now give $\dimension (\pi (\supp E_1))=1$, and so $E_1 \in \mathcal{A}^2(\pi)_1 \cap W_{1,\Phi}$.  By Properties A3 and D2 again, we know  $\dim(\wh{E_1})$ is either 1 or 2.  If $\dimension (\wh{E_1})=1$, then Lemma \ref{Erik} says $E_1 \in \Tc_{31} \subset \Tc_{40}$; if $\dimension (\wh{E_1})=2$, then $E_1 \in \Tc_{32} \subset \Tc_{40}$.  

On the other hand, Property D2 implies $\dimension (\pi (\supp E_0))=1$, and so  $E_0 \in \Ac^2(\pi)_1 \cap W_{0,\Phi} = C_{30} \subset \Tc_{40}$.  Hence $E \in \Tc_{40}$ as wanted.
\end{proof}

\begin{lem}\label{Jam} 
$\mathcal{A}^{\leq 2}_X \cap W_{0,\Phi} =\mathcal{F}_4$, and it is a torsion class in $\Ac_X$.
\end{lem}

\begin{proof}
Since $\Ac^{\leq 2}_X$ and $W_{0,\Phi}$ are both torsion classes in $\Ac_X$ by Lemma \ref{lem:C00isTC} and Property TC3, their intersection is also a torsion class in $\Ac_X$.  Hence it remains to show the equality $\mathcal{A}^{\leq 2}_X \cap W_{0,\Phi} =\mathcal{F}_4$.  The inclusion $\mathcal{F}_4 \subseteq \mathcal{A}^{\leq 2}_X \cap W_{0,\Phi}$ is clear from the definition of $\Fc_4$ and Properties C0 and C2.

To show the inclusion $\mathcal{A}^{\leq 2}_X \cap W_{0,\Phi} \subseteq \Fc_4$, take any $E \in \mathcal{A}^{\leq 2}_X \cap W_{0,\Phi}$.  Note that $\Fc_3$ is  a torsion class in $\Ac_X$ by Lemma \ref{Jamie}; let $\Fc$ denote the corresponding torsion-free class in $\Ac^{\leq 2}_X$. Then we have an $\Ac^{\leq 2}_X$-short exact sequence
\[
0 \to E' \to E \to E'' \to 0
\]
where $E' \in \Fc_3 \subset \Fc_4$ and $E'' \in \Fc$.  Since $W_{0, \Phi}$ is a torsion class in $\Ac_X$ (Property TC3), we have $E'' \in \Ac_X^{\leq 2}\cap W_{0,\Phi}$.  We will be done once we show $E'' \in \Fc_4$.

If $\dimension E'' \leq 1$, then $E'' \in \Fc_2 \subset \Fc_4$ by Lemma \ref{Max}, so we can assume $\dimension E'' = 2$.  Then Property D2 implies $\dimension (\pi (\supp E''))$ is 1 or 2.  If $\dimension (\pi (\supp E''))=1$, then $E'' \in C_{30} \subset \Fc_4$, so we can further assume $\dimension (\pi (\supp E''))=2$.  

Properties A3 and D2 together now imply $\dimension \wh{E''}$ is either 2 or 3.  If $\dimension \wh{E''}=2$, then Lemma \ref{Daniel} gives $\dimension (\pi (\supp E'')) \leq 1$, a contradiction.  If $\dimension \wh{E''}=3$, then $E'' \in C_{40} \subset \Fc_4$, completing the proof.
\end{proof}

\begin{lem}\label{Veronica} 
$W_{0,\Phi} =\mathcal{F}_5$.
\end{lem}

\begin{proof}
By Lemma \ref{Jam}, we have $\Fc_4 \subseteq W_{0,\Phi}$; since $C_{50} \subset W_{0,\Phi}$ by definition, we have $\Fc_5 \subseteq W_{0,\Phi}$.  To show the other inclusion, take any $E \in W_{0,\Phi}$.    Since $\Fc_4$ is a torsion class in $\Ac_X$ (Lemma \ref{Jam}), we have an $\Ac_X$-short exact sequence
\[
0 \to E' \to E \to E'' \to 0
\]
where $E' \in \Fc_4 \subset \Fc_5$ and $\Hom (\Fc_4, E'')=0$.  Since $W_{0,\Phi}$ is a torsion class in $\Ac_X$, we have $E'' \in W_{0,\Phi}$.  If $\dimension E'' \leq 2$, then $E'' \in \Fc_4$ by Lemma \ref{Jam} again, and so $E''=0$; if $\dimension E'' = 3$, then Property CH2p implies $\dimension (\pi (\supp E''))=2$, i.e.\ $E'' \in C_{50} \subset \Fc_5$.  Overall, we have $E \in \Fc_5$.
\end{proof}

\begin{lem}\label{Joshua}
$\mathcal{T}_{50}$ is a torsion class in $\Ac_X$.
\end{lem}

\begin{proof}
We have $\mathcal{T}_{50}= \langle \mathcal{T}_{40}, C_{50}\rangle$ by definition.  Take any $E \in C_{50}$ and any $\Ac_X$-surjection $E \twoheadrightarrow E'$.  By Lemmas \ref{Esteban} and \ref{Bill}, it suffices to show $E' \in \Tc_{50}$.  Since $E \in \Fc_5$, Lemma \ref{Veronica} and Property TC3 together give $E' \in\Fc_5 \subset \Tc_{50}$, completing the proof.
\end{proof}

\begin{lem} \label{Lin}
$\mathcal{T}_{51}$ is a torsion class in $\Ac_X$.
\end{lem}

\begin{proof}
We have $\mathcal{T}_{51}= \langle \mathcal{T}_{50}, C_{51}\rangle$ by definition.  Take any $E \in C_{51}$ and any $\Ac_X$-surjection $E \twoheadrightarrow E'$.  By Lemmas \ref{Joshua} and \ref{Bill}, it suffices to show $E' \in \Tc_{51}$.  By Property TC3, we have an $\Ac_X$-short exact sequence
\[
0 \to A_0 \to E' \to A_1 \to 0
\]
where $A_i \in W_{i,\Phi}$.  Then $A_0 \in \Fc_5 \subset \Tc_{51}$ by Lemma \ref{Veronica}, so it remains to show $A_1 \in \Tc_{51}$.

If $\dimension A_1 \leq 2$, then  $A_1 \in \Tc_{40} \subset \Tc_{51}$ by Lemma \ref{Esteban}; so let us assume $\dimension A_1 = 3$ from now on.  Since $E \in C_{51}$, we have $E \in W_{1,\Phi}$.  Thus by Lemma \ref{Gabriel}, there is an $\Ac_X$-surjection $\wh{E} \twoheadrightarrow \wh{A_1}$.   By definition of $C_{51}$, we have $\dimension \wh{E}=2$, and so $\dimension \wh{A_1} \leq 2$ by Property D1.  

Suppose $\dimension \wh{A_1} \leq 1$; then $\dimension (\pi (\supp \wh{A_1})) \leq 1$ by Property D2, and so $\dimension (\pi (\supp A_1)) = \dimension (\pi (\supp \wh{A_1})) \leq 1$ by Property A3.  By Property D2 again, however, we obtain $\dimension A_1 \leq 2$, a contradiction.  Thus we must have $\dimension \wh{A_1}=2$, meaning $A_1 \in C_{51} \subset \Tc_{51}$, and we are done.
\end{proof}

\begin{lem}\label{lem:last}
$\Ac_X = \mathcal{T}_{52}$.
\end{lem}

\begin{proof}
Take any $E \in \Ac_X$. Property TC3 asserts there is an $\Ac_X$-short  exact sequence 
\[
 0\to E_0 \to E \to E_1 \to 0
 \]
where $E_i \in W_{i,\Phi}$.  Note that $E_0 \in \Fc_5 \subset \Tc_{52}$ by Lemma \ref{Veronica}, so it remains to show $E_1 \in \Tc_{52}$.  

If $\dimension E_1 \leq 2$, then $E_1 \in \Tc_{40}\subset \Tc_{52}$ by Lemma \ref{Esteban}, so let us assume $\dimension E_1 = 3$.  Then $\dimension (\pi (\supp E_1))=2$  by Property D2 and the fact that the dimension of the base $B$ of the elliptic fibration $\pi : X \to B$ is 2.  Properties A4 and D2 together then gives $\dimension \wh{E_1} \geq 2$.  Then $E_1 \in C_{51} \subset \Tc_{52}$ if $\dimension \wh{E_1}=2$, while $E_1 \in C_{52} \subset \Tc_{52}$ if $\dimension \wh{E_1}=3$. This completes the proof of the lemma.
\end{proof}

\begin{proof}[Proof of Theorem \ref{thm:main}]
This follows from Lemmas \ref{lem:C00isTC}, \ref{lem:C00C20isTC}, \ref{Jade}, \ref{Lucas}, \ref{Sarah}, \ref{Max}, \ref{Tyler}, \ref{Jamie}, \ref{Zuly}, \ref{Gaby}, \ref{Manuel}, \ref{Esteban}, \ref{Jam}, \ref{Veronica}, \ref{Joshua} and \ref{Lin}.
\end{proof}

\bibliography{refs}{}
\bibliographystyle{plain}

\end{document}